\documentclass[12pt]{article}
\usepackage[english]{babel}

\usepackage{amsmath}
\usepackage{amssymb,amsthm}

\usepackage{bbm}

\numberwithin{equation}{section}

\newtheorem{theorem}{Theorem}

\newtheorem{lemma}{Lemma}[section]

\newtheorem{remark}[lemma]{Remark}
\newtheorem{proposition}{Proposition}
\newtheorem{notation}[lemma]{Notation}

\renewcommand{\leq}{\leqslant}
\renewcommand{\geq}{\geqslant}

\newcommand{\Rot}{\mathcal{R}}

\makeatletter
\newcommand*{\IfItalic}{%
  \ifx\f@shape\my@test@it
    \expandafter\@firstoftwo
  \else
    \expandafter\@secondoftwo
  \fi
}
\newcommand*{\my@test@it}{it}
\makeatother

\newcommand{\myae}{\IfItalic{\emph{\mbox{\ae}}}{\mbox{\ae}}}

\usepackage[all]{xy}

\usepackage{amscd}

\begin{document}

\begin{center}
{\Large Self-semiconjugation\\
of piecewise linear unimodal maps }

{\large Makar Plakhotnyk\\
University of S\~ao Paulo, Brazil.\\

makar.plakhotnyk@gmail.com}\\
\end{center}

\begin{abstract}
We devote this work to the functional equation $\psi \circ g =
g\circ \psi$, where $\psi$ is an unknown function and $g$ is
piecewise linear unimodal map, which is topologically conjugated
to the tent map. We will call such $\psi$ self-semiconjugations of
$g$. Our the main results are the following:

1. Suppose that there is a self-semiconjugation of $g$, whose
tangent at $0$ is not a power of $2$, and suppose that all the
kinks of $g$ are in the complete pre-image of $0$. Then all the
self-semiconjugations of $g$ are piecewise linear.

2. Suppose that all self-semiconjugations of $g$ are piecewise
linear. Then the conjugacy of $g$ and the tent map is piecewise
linear~\footnote{This work is partially supported by FAPESP (S\~ao
Paulo, Brazil).}~\footnote{AMS subject classification:
37E05  
}.
\end{abstract}

\section{Introduction}

We continue in this work the study of the topological conjugation
of piecewise linear unimodal maps, which was started in our
previous works~\cite{UMZh-2016}, \cite{Plakh-Arx-Odesa}
and~\cite{Plakh-Arx-Derivat-6}.

We call a map $g:\, [0, 1]\rightarrow [0, 1]$ a
\underline{\emph{unimodal}} map, if it is continuous and is of the
form $$ g(x) = \left\{
\begin{array}{ll}g_l(x),& 0\leq x\leq
v,\\
g_r(x), & v\leq x\leq 1,
\end{array}\right.
$$ where %
$v\in (0,\, 1)$, the function $g_l$ increase, the function $g_r$
decrease, and $$g(0)=g(1)=1-g(v)=0.$$

The most studied and popular example of unimodal map is the
\underline{\emph{tent map}}, i.e. $f:\, [0, 1]\rightarrow [0, 1]$
of the form
\begin{equation}\label{eq:1.1} x \mapsto 2x-|1-2x|
\end{equation}

\begin{theorem}\label{th:1}\cite[p. 53]{Ulam-1964-a} %
The tent map~\eqref{eq:1.1} is topologically conjugated to the
unimodal map $g$ if and only if the complete pre-image of $0$
under the action of $g$ is dense in $[0,\, 1]$.
\end{theorem}

Remind, that the set $g^{-\infty}(a) = \bigcup\limits_{n\geq
1}g^{-n}(a)$, where $g^{-n}(a) = \left\{ x\in [0,\, 1]:\, g^n(x) =
a\right\}$ for all $n\geq 1$, is called the complete pre-image of
$a$ (under the action of the map $g$).

We will call the map, defined on the set $A\subseteq \mathbb{R}$,
with values in $\mathbb{R}$, \underline{\emph{linear}}, if its
graph is a line (or a line segment). We will call a map
\underline{\emph{piecewise linear}}, if its domain can be divided
to finitely many intervals, such that the map is linear on each of
them. A point, where the piecewise linear map is not
differentiable, will be called a \underline{\emph{kink}} of this
map.

We will call piecewise linear unimodal map a
\underline{\emph{carcass map}}. We will say that a carcass map is
\underline{\emph{firm}}, if all its kinks are in the complete
pre-image of $0$.

We have studied the properties of the firm carcass maps
in~\cite{Plakh-Arx-Derivat-6}. Precisely, by~\cite[Theorem
2]{Plakh-Arx-Derivat-6}, the complete pre-image of $0$ under the
action of firm carcass map is dense in $[0,\, 1]$, whence, by
Theorem~\ref{th:1}, any firm carcass map is topologically
conjugated to the tent map.

We have described in~\cite{Plakh-Arx-Odesa} the continuous
solutions $\xi:\, [0, 1]\rightarrow [0, 1]$ of the functional
equation
\begin{equation}\label{eq:1.2}
\xi\circ f = f\circ \xi,
\end{equation} where $f$ is the tent map.

\begin{proposition}\label{prop:1}\cite[Theorem 1]{Plakh-Arx-Odesa}
1. Let $\xi$ be an arbitrary continuous solution of the functional
equation~\eqref{eq:1.2}. Then $\xi$ is one of the following forms:

a. There exists $t\in \mathbb{N}$ such that
\begin{equation}\label{eq:1.3}
\xi = \xi_t:\, x \mapsto \displaystyle{\frac{1 - (-1)^{[tx]}}{2}
+(-1)^{[tx]}\{tx\}},\end{equation} where $\{\cdot \}$ denotes the
function of the fractional part of a number and $[\cdot ]$ is the
integer part.

b. $\xi(x)=x_0$ for all $x$, where either $x_0=0$, or $x_0= 2/3$.

\noindent 2. For every $t\in \mathbb{N}$ the
function~\eqref{eq:1.3} satisfies~\eqref{eq:1.2}.
\end{proposition}

The simplicity of the description of the solution
of~\eqref{eq:1.2} motivates the study of the functional equation
\begin{equation}\label{eq:1.4} \psi \circ g = g\circ \psi,
\end{equation} %
for a unknown (continuous; surjective) $\psi:\, [0, 1]\rightarrow
[0, 1]$ and given carcass map $g$.

The main results of this work is the next two theorems.

\begin{theorem}\label{th:2}
Let $g$ be a firm carcass map, which is topologically conjugated
to the tent map. Suppose that there is a piecewise linear solution
of~\eqref{eq:1.4} with tangent $t$ at $0$, such that $t$ is not a
natural power of $2$. Then all the solutions of~\eqref{eq:1.4} are
piecewise linear.
\end{theorem}

\begin{theorem}\label{th:3}
Let $g$ be a carcass map, which is topologically conjugated to the
tent map. If all the solutions of~\eqref{eq:1.4} are piecewise
linear, then the conjugacy of $g$ and the tent map is piecewise
linear.
\end{theorem}

Thus, piecewise linear conjugacy of carcass maps naturally appears
in the study of the properties of conjugated carcass maps and firm
carcass maps. Remark that for any increasing piecewise linear
homeomorphism $h:\, [0, 1]\rightarrow [0, 1]$ we can construct a
carcass map $g:\ [0, 1]\rightarrow [0, 1]$ by formula
\begin{equation}\label{eq:1.5}g = h \circ f \circ h^{-1},
\end{equation} where $f$ %
is the tent map. Clearly, the constructed map $g$ is piecewise
linear as composition of linear maps, and will, by the
construction, satisfy
 \begin{equation}\label{eq:1.6} h \circ f
= g\circ h.\end{equation} Since $h(0)=0$, then directly
by~\eqref{eq:1.5} obtain that $g(0)=g(1)=0$. Conclude that the
constructed $g$ is a carcass map, because a topological
conjugation preserves the monotonicity of maps.

Thus, there are ``a lot'' of carcass maps, which are topologically
conjugated to the tent-map via the piecewise linear conjugacy,
precisely, any increasing piecewise linear conjugacy gives such
map. From another hand, the following facts show that the
piecewise linearity of the conjugacy of a carcass map and the tent
map is a ``rare'' property.

\begin{proposition}\cite[Lemmas~1 and~13]{UMZh-2016}
Let the tent map be topologically conjugated with the carcass map
$g$ via a piecewise linear conjugacy. Denote $x_0$ the positive
fixed point of $g$. Then $g'(0)=2$ and $g'(x_0-)\cdot g'(x_0+)
=4$.~\footnote{Here and below we denote $g'(x_0-)$ and $g'(x_0+)$
the left and the right derivative of $g$ at $x_0$.}
\end{proposition}

\begin{proposition}\cite[Theorem~2]{UMZh-2016}
For any $v\in (0, 1)$ let $\widehat{g}$ be increasing piecewise
linear map $\widehat{g}:\, [0, v]\rightarrow [0, 1]$ such that
$\widehat{g}(0) = 1-\widehat{g}(v) = 0$ and $\widehat{g}\,
'(0)=2$. There exists the unique carcass map $g$, which coincides
with $\widehat{g}$ on $[0, v]$ and is topologically conjugated
with a tent map via a piecewise linear conjugacy.
\end{proposition}

\begin{proposition}\cite[Theorem~3]{UMZh-2016}
For any $v\in (0, 1)$ let $\widehat{g}$ be decreasing piecewise
linear map $\widehat{g}:\, [v, 1]\rightarrow [0, 1]$ such that
$\widehat{g} (v) = 1-\widehat{g}(1) = 1$ and $\widehat{g}\,
'(x_0-)\cdot \widehat{g}\, '(x_0+) =4$, where $x_0$ is the fixed
point of $\widehat{g}$. There exists the unique carcass map $g$,
which coincides with $\widehat{g}$ on $[v, 1]$ and is
topologically conjugated with a tent map via a piecewise linear
conjugacy.
\end{proposition}

\section{Properties of unimodal maps}

\subsection{The connection of topological conjugacy and semi conjugacy}

\begin{lemma}\label{lema:2.1}
Suppose that for a unimodal map $g:\, [0, 1]\rightarrow [0, 1]$,
and the conjugacy $h: [0, 1]\rightarrow [0, 1]$ the
equality~\eqref{eq:1.6} holds, where $f$ is the tent-map. For any
(continuous; surjective)  solution $\psi:\, [0, 1]\rightarrow [0,
1]$ of~\eqref{eq:1.4} there exists a (continuous; surjective)
solution $\xi$ of~\eqref{eq:1.2} such that
\begin{equation}\label{eq:2.1} \psi = h\circ \xi\circ h^{-1}.
\end{equation}
\end{lemma}

\begin{proof}
As we have mentioned in~\cite{Plakh-Arx-Odesa}, the description of
the non-homeomorphic (surjective) solutions $\eta:\, [0,
1]\rightarrow [0, 1]$ of the functional equation
\begin{equation}\label{eq:2.2} \eta \circ f = g\circ
\eta,\end{equation} can be reduced to the description of the
non-homeomorphic (surjective) solutions $\xi:\, [0, 1]\rightarrow
[0, 1]$ of the functional equation~\eqref{eq:1.2}. Indeed, it is
clear from the commutative diagram $$ \xymatrix{ [0,\, 1]
\ar@/_3pc/@{-->}_{\xi}[dd] \ar^{f}[rr]
\ar_{\eta}[d] && [0,\, 1] \ar^{\eta}[d] \ar@/^3pc/@{-->}^{\xi}[dd]\\
[0,\, 1] \ar^{g}[rr] \ar^{h^{-1}}[d] && [0,\, 1] \ar_{h^{-1}}[d]\\
[0,\, 1] \ar^{f}[rr] && [0,\, 1] }$$ that there is one-to-one
correspondence
\begin{equation}\label{eq:2.3} \left\{\begin{array}{l}\xi  = h^{-1}\circ \eta,\\
\eta = h\circ \xi
\end{array}\right.\end{equation}
between the solutions $\eta$ of~\eqref{eq:2.2} and the solutions
$\xi$ of~\eqref{eq:1.2}.

For any continuous (surjective) solution $\psi$ of~\eqref{eq:1.4},
and the homeomorphism $h$, which satisfies~\eqref{eq:1.6}, write
the commutative diagram $$ \xymatrix{ [0,\, 1]
\ar@/_3pc/@{-->}_{\psi}[dd] \ar^{g}[rr]
\ar^{h^{-1}}[d] && [0,\, 1] \ar_{h^{-1}}[d] \ar@/^3pc/@{-->}^{\psi}[dd]\\
[0,\, 1] \ar^{f}[rr] \ar_{\eta}[d] && [0,\, 1] \ar^{\eta}[d]\\
[0,\, 1] \ar^{g}[rr] && [0,\, 1] }$$ which, similarly
to~\eqref{eq:2.3}, defines one-to-one correspondence between the
solutions of the functional equations~\eqref{eq:2.2}
and~\eqref{eq:1.4}.

\begin{equation}\label{eq:2.4}
\left\{\begin{array}{l}\psi  = \eta\circ h^{-1},\\
\eta = \psi\circ h.
\end{array}\right.\end{equation}
Now~\eqref{eq:2.1} follows from~\eqref{eq:2.3} and~\eqref{eq:2.4}.
\end{proof}

\subsection{Pre-images of $0$ under a carcass map and topological conjugation}

The next fact is the classical property of the tent map.

\begin{remark}\label{rem:2.2}
For every $n\geq 1$ we have that $$f^{-n}(0) = \left\{
\frac{k}{2^{n-1}},\, 0\leq k\leq 2^{n-1}\right\},$$ where $f$ is
the tent map.
\end{remark}

Let a carcass map $g$ be fixed till the end of this subsection.
By~\cite[Remark~3.1]{Plakh-Arx-Derivat-6}, the set $g^{-n}(0)$
consists of $2^{n-1}+1$ points.
Following~\cite[Notation~3.2]{Plakh-Arx-Derivat-6}, for any $n\geq
1$ denote
\begin{equation}\label{eq:2.5}
g^{-n}(0) = \{\mu_{n,k}(g),\, 0\leq k\leq 2^{n-1}
\}.\end{equation} where $\mu_{n,k}(g)<\mu_{n,k+1}(g)$ for all
$k,\, 0\leq k<2^{n-1}$.

\begin{remark}\cite[Remark~3.3]{Plakh-Arx-Derivat-6}\label{rem:2.3}
Notice that $\mu_{n,k}(g) = \mu_{n+1,2k}(g)$ for all $k,\, 0\leq k
\leq 2^{n-1}$.
\end{remark}

\begin{lemma}\cite[Lemma~3.6]{Plakh-Arx-Derivat-6}\label{lema:2.4}
For any unimodal maps $g_1,\, g_2:\ [0, 1]\rightarrow [0, 1]$ and
the conjugacy $h$, which satisfies
\begin{equation}\label{eq:2.6}
h\circ g_1 = g_2\circ h,
\end{equation} the equality $$ h(\mu_{n,k}(g_1)) =\mu_{n,k}(g_2)
$$ holds for all $n\geq 1$ and $k,\, 0\leq k\leq
2^{n-1}$.
\end{lemma}

Moreover, by~\cite[Theorem~7]{Plakh-Arx-Derivat-6}, the conjugacy
is unique in the case of Lemma~\ref{lema:2.4}.

For every $n\geq 1$ and $k,\, 0\leq k<2^{n-1}$ write $\mu_{n,k}$
instead of $\mu_{n,k}(g)$.

\begin{notation}\cite[Notation~4.7]{Plakh-Arx-Derivat-6}\label{not:2.5}
For any $n\geq 1$ and $k,\, 0\leq k< 2^{n-1}$ denote $I_{n,k} =
(\mu_{n,k},\, \mu_{n,k+1})$ and for any interval $I = (a,\, b)$
denote $\# I = b -a$.
\end{notation}

\begin{notation}\cite[Notation~4.9]{Plakh-Arx-Derivat-6}
For any $n\geq 1$ and $k,\, 0\leq k< 2^n$ denote
$$\delta_{n,k} =
\frac{\mu_{n+1,2k+1}-\mu_{n,k}}{\mu_{n,k+1}-\mu_{n,k}} = \frac{\#
I_{n+1,2k}}{\# I_{n,k}}.$$
\end{notation}

\subsection{Pre-images of $0$ under a firm carcass map}\label{sec:2.3}

Let $g$ be a firm carcass map, which will be fixed till the end of
this section. Denote by $n_0$ the minimal natural number such that
$g^{n_0}(x) = 0$ for each kink $x$ of $g$.

Write $\mu_{n,k}$ instead of $\mu_{n,k}(g)$ for all $n\geq 0$ and
$k,\, 0\leq k\leq 2^{n-1}$.

\begin{notation}
For every $a\in \{0; 1\}$ denote $$ \Rot(a) = 1-a.
$$
\end{notation}

\begin{remark}\cite[Remark~4.19]{Plakh-Arx-Derivat-6}\label{rem:2.8}
Let $n\geq n_0$ and $k,\, 0\leq k< 2^{n-1}$ have the binary
expansion
$$
k = \sum\limits_{i=0}^{n}x_i2^{n-i} $$ for all $n\geq 1$. %
Then $\delta(I_{n+1,k}) =
\Rot^{x_{n-n_0+1}}(\delta(I_{n_0,p_n})),$ where $$p_{i} =
\sum\limits_{j=i-n_0+2}^{i}\Rot^{x_{i+1-n_0}}(x_j)2^{\, i-j}
$$ for all $i,\, n_0+1\leq i\leq n$.
\end{remark}

It follows from Remark~\ref{rem:2.8} that for any $n\geq n_0$ and
$k,\ 0\leq k<2^{n-1}$ the number $\delta(I_{n,k})$ depends on the
last $n_0$ binary digits of $k$ and is independent on $n$ and
other digits of $k$.

There is an important corollary from Remark~\ref{rem:2.8}.

\begin{remark}\label{rem:2.9}
There are $l_0,\ldots, l_{2^{n_0}-1}$ with the following property.
For every $n\geq n_0$ and every $k,\, 0\leq k<2^{n-1}$ let
$k_0\geq 0$ and $j,\, 0\leq j<2^{n_0}$ be such that $k =
2^{n_0}\cdot k_0 +j$. Then
$$ \# I_{n,k} = l_j\cdot \# I_{n,2^{n_0}\cdot k_0}. $$
\end{remark}

\begin{notation}
1. For any $k,\, 0\leq k<2^{n_0}$ denote $$ \delta_k =
\delta(I_{n_0,k}).
$$

2. For any $k\geq 2^{n_0}$ denote by $$ \delta_k =
\delta(I_{n_0,k^*}),$$ where the last $n_0$ binary digits of $k$
and $k^*$ coincide and $0\leq k^*< 2^{n_0}$.
\end{notation}

The next fact follows from Remark~\ref{rem:2.8}.

\begin{remark}\label{rem:2.11}
Notice that $\delta_k = \delta_{k +2^{n_0}}$ for all $k\geq 0$.
\end{remark}

\begin{remark}\label{rem:2.12}
For every $n\geq n_0$ and $k,\, 0\leq k<2^{n-1}$ we have that $$
\# I_{n+n_0,i\cdot 2^{n_0}} =\frac{\#
I_{n,i}}{\sum\limits_{k=1}^{2^{n_0}-1}l_k}.
$$
\end{remark}

\begin{proof}
Since $$ \overline{I_{n,i}} =
\bigcup\limits_{k=0}^{2^{n_0}-1}\overline{I_{n+n_0,i\cdot 2^{n_0}
+k}},
$$ then, by Remark~\ref{rem:2.9}, $$
\# I_{n,i} = \sum\limits_{k=0}^{2^{n_0}-1}l_k\cdot
\#I_{n+n_0,i\cdot 2^{n_0}}
$$
and the fact follows.
\end{proof}

\subsection{Self-semiconjugation of carcass maps}

Let a carcass map $g$, such that the complete pre-image of $0$
under the action of $g$ is dense in $[0, 1]$, be fixed till the
end of this section. By~\cite[Theorem 7]{Plakh-Arx-Derivat-6} let
$h:\, [0, 1]\rightarrow [0, 1]$ be the unique conjugacy, such
that~\eqref{eq:1.6} holds.

\begin{remark}\label{rem:2.13}
The map $\xi_t$, defined by~\eqref{eq:1.3}, can be expressed by
the formula $$ \xi_t:\, x\rightarrow tx
$$ for all $x\in \left[0, \frac{1}{t}\right]$.
\end{remark}

\begin{remark}\label{rem:2.14}
For the map $\xi_t$, defined by~\eqref{eq:1.3}, $$
\xi_t(\mu_{n,k}(f)) = \mu_{n,tk}(f),$$ whenever $k\leq \left[
\frac{2^{n-1}}{t}\right]$.
\end{remark}

\begin{proof}
The remark follows from Remarks~\ref{rem:2.2} and~\ref{rem:2.13}.
\end{proof}

\begin{lemma}\label{lema:2.15}
For any self-semiconjugation $\psi$ of $g$ there is $t\in
\mathbb{N}$ such that $$ \psi(\mu_{n,k}(g)) = \mu_{n,kt}(g)
$$ for all $n\geq 1$ and $k\leq \left[
\frac{2^{n-1}}{t}\right]$.
\end{lemma}

\begin{proof} Let $\psi$ be a self-semiconjugation of $g$, and $\xi$ be the
self-semiconjugation of the tent map, whose existence follows from
Lemma~\ref{lema:2.1}. By Proposition~\ref{prop:1} the map $\xi$ is
$\xi = \xi_t$ of the form~\eqref{eq:1.3} for some $t\geq 1$ and
\begin{equation}\label{eq:2.7} \psi = h\circ \xi_t \circ
h^{-1}.\end{equation}

Write $\psi_t$ for the function $\psi$, which is defined
by~\eqref{eq:2.7}.

By~\cite[Lemma~3.6]{Plakh-Arx-Derivat-6} we have that
\begin{equation}\label{eq:2.8} h(\mu_{n,k}(f)) = \mu_{n,k}(g)
\end{equation} for all $n\geq 1$ and $k,\, 0\leq k\leq 2^{n-1}$, where
$\mu_{n,k}$ are defined in~\eqref{eq:2.5}.

For all $n\geq 1$ and $k\leq \left[ \frac{2^{n-1}}{t}\right]$ it
follows from Remark~\ref{rem:2.14} that

$$ \psi_t(\mu_{n,k}(g)) \stackrel{\text{\eqref{eq:2.7}}}{=} %
(h\circ \xi_t)( h^{-1}(\mu_{n,k}(g) ))
\stackrel{\text{\eqref{eq:2.8}}}{=}$$
$$
=(h\circ \xi_t)( \mu_{n,k}(f) ) =h( \xi_t(\mu_{n,k}(f) ))
\stackrel{\text{Rem.~\ref{rem:2.14}}}{=}
$$$$
=h( \mu_{n,kt}(f) ) \stackrel{\text{\eqref{eq:2.8}}}{=}
\mu_{n,kt}(g),
$$ and lemma follows.\end{proof}

\subsection{Conjugation of carcass maps via piecewise linear conjugacy}

\begin{lemma}\label{lema:2.16}
Let $g_1$ and $g_2$ be unimodal maps, which are topologically
conjugated to the tent map. If a continuous solution $h$
of~\eqref{eq:2.6} is linear on some interval, then $h$ is
piecewise linear in the entire $[0, 1]$.
\end{lemma}

\begin{proof}
Since $g_1$ and $g_2$ are topologically conjugated to the tent
map, then, by Theorem~\ref{th:1}, there are $n$ and $k_n^-,\,
k_n^+\in \{0,\ldots, 2^{n-1}\}$ such that $h$ is linear linear on
$[\mu_{n,k_n^-}(g_1), \mu_{n,k_n^+}(g_1)]$.

The formula~\eqref{eq:2.6} defines $h$ on the interval $g_1\circ
[\mu_{n,k_n^-}, \mu_{n,k_n^+}]$. Since $g_1\circ g_1^{-n}(0) =
g^{-n+1}(0)$, then there are $k_{n-1}^-,\, k_{n-1}^+\in
\{0,\ldots, 2^{n-2}\}$ such that $h$ is piecewise linear linear on
$[\mu_{n-1,k_{n-1}^-}(g_1), \mu_{n-1,k_{n-1}^+}(g_1)] = g_1\circ
[\mu_{n,k_n^-}, \mu_{n,k_n^+}]$.

In this manner lemma follows by induction on $n$, sice $[0,\, 1] =
[\mu_{1,0}(g_1), \mu_{1,1}(g_1)]$.
\end{proof}

\begin{lemma}\label{lema:2.17}
Let $g$ be a carcass map, which is topologically conjugated to the
tent map $f$. The following conditions are equivalent.

1. The conjugacy of $f$ and $g$ is piecewise linear.

2. There is $a\in (0,\, 1)$ and $r, t>0$ such that for every
$n\geq 1$ and $k,\, 0\leq k\leq 2^{n-1}$ the equality $$
\mu_{n,k}(g) = \frac{tk}{2^{n-1}}
$$ holds, whenever $\mu_{n,k}(g)< r$.
\end{lemma}

\begin{proof}
It follows from Lemma~\ref{lema:2.4} and Remark~\ref{rem:2.2} that
condition 1. is equivalent to that the conjugacy of $f$ and $g$ is
linear in some neighborhood of $0$.

Now lemma follows from Lemma~\ref{lema:2.16}.
\end{proof}

\begin{lemma}\label{lema:2.18}
Let a carcass map $g$ be linear on some interval $(a,\, b)$. Then
$$
\# (g\circ (a,\, b)) = g'(a)\cdot \# (a,\, b).
$$
\end{lemma}

\begin{proof}
This fact is trivial.
\end{proof}

\section{Self-semiconjugations}

\subsection{Carcass maps}

\begin{lemma}\label{lema:3.1}
Let $g$ be a carcass map and $a\in (0, 1)$ be the first positive
kink of $g$. If $\psi\, '(0)>g'(0)$ for a piecewise linear
self-semiconjugation $\psi$ of $g$, then $\frac{a\cdot
g'(0)}{\psi\, '(0)}$ is the first positive kink of $\psi$.
\end{lemma}

\begin{proof}
It follows from Lemma~\ref{lema:2.15} that $\psi\, '(0)>1$.

Since $\psi$ is piecewise linear, then there is $\varepsilon>0$
such that $\psi(x)=\psi\, '(0)\cdot x$ for all $x\in (0,
\varepsilon)$.

Suppose that $\varepsilon>0$ is such that $\psi\, '(0)\cdot
\varepsilon <a$. Notice that in this case $x<a$ for all $x\in (0,
\varepsilon)$, because $\psi\, '(0)>1$. Thus, for all $x \in (0,
\varepsilon)$ we have that $\psi\, '(x)=\psi\, '(0)\cdot x$,
$(g\circ \psi)(x) = g'(0)\cdot \psi\, '(0)\cdot x$ and
$g(x)=g'(0)\cdot x$. Now, by~\eqref{eq:1.4}, $$ \psi(x) = \psi\,
'(0)\cdot x
$$ for all $x\in g\circ (0,
\varepsilon).$

Thus, for every $x\in g\circ \left( 0, \frac{a}{\psi\,
'(0)}\right)$ we have that $\psi(x) = \psi\, '(0)\cdot x$. Notice
that $g\circ \left( 0, \frac{a}{\psi\, '(0)}\right) = \left( 0,
\frac{a\cdot g'(0)}{\psi\, '(0)}\right)$. Remark that
$\frac{a\cdot g'(0)}{\psi\, '(0)}< a$.

Take an arbitrary $\delta,\, 0<\delta <a\cdot (g'(0)-1)$ is such
that $g$ is linear on $(a,\, a+\delta)$. Then for every $x\in
\left( \frac{a+\delta}{\psi\, '(0)}, \frac{a\cdot g'(0)}{\psi\,
'(0)}\right)$ we have that $\psi(x)=\psi\, '(0)\cdot x$, because
$x<\frac{a\cdot g'(0)}{\psi\, '(0)}$; also $g(x) = g'(0)\cdot x$,
because $x<a$, and $(g\circ \psi)(x) = g'(a+)\cdot (\psi\,
'(0)\cdot x -a) + g(a)$.

Now, it follows from~\eqref{eq:1.4} that
\begin{equation}\label{eq:3.1} \psi(g'(0)\cdot x) =g'(a+)\cdot
(\psi\, '(0)\cdot x -a) + g(a).
\end{equation} Denote %
$u = g'(0)\cdot x$ and remark that $u\in \left(
\frac{(a+\delta)\cdot g'(0)}{\psi\, '(0)}, \frac{a\cdot
(g'(0))^2}{\psi\, '(0)}\right)$, whenever $x\in \left(
\frac{a+\delta}{\psi\, '(0)}, \frac{a\cdot g'(0)}{\psi\,
'(0)}\right)$. Rewrite~\eqref{eq:3.1} as $$ \psi(u) = g'(a+)\cdot
\left(\frac{\psi\, '(0)\cdot u}{g'(0)} -a\right) + g(a).
$$
Clearly, if $g'(0)\neq g'(a+)$, then there is $u\in \left(
\frac{(a+\delta)\cdot g'(0)}{\psi\, '(0)}, \frac{a\cdot
(g'(0))^2}{\psi\, '(0)}\right)$ such that $\psi(u) \neq \psi\,
'(0)\cdot u$.
Remark that %
if $\delta \approx 0$, then $\frac{(a+\delta)\cdot g'(0)}{\psi\,
'(0)} \approx \frac{a\cdot g'(0)}{\psi\, '(0)}$.
\end{proof}

\begin{lemma}\label{lema:3.2}
Suppose that for any $t\geq 1$ the map $\psi_t$, defined
by~\eqref{eq:2.7}, is piecewise linear. Then there is $w\in
\mathbb{R}$ such that for every $n\geq 1$ and $k$, the equality
$$ \mu_{n,k}(g) =
\frac{w\cdot k}{2^{n-1}},
$$ holds whenever $\mu_{n,k}(g)\leq g(a)$, where $a\in (0, 1)$ is
the first positive kink of $g$.
\end{lemma}

\begin{proof}
For any $n\geq 1$ and $k,\, 0\leq k\leq 2^{n-1}$ we will write
$\mu_{n,k}$ instead of $\mu_{n,k}(g)$.

Denote $\tau_t$ the tangent of $\psi_t$ at $0$. Then, by
Lemma~\ref{lema:3.1}, $$ \tau_t = \frac{\mu_{n,kt}}{\mu_{n,1}}.
$$

Notice that $\tau_2 = g'(0)$ and $\mu_{n,k} = \tau_2\cdot
\mu_{n+1,k},$ whenever $\mu_{n,k}<a$.

Remark that, if $\mu_{n,k}\leq g(a)$, then
\begin{equation}\label{eq:3.2} \mu_{n,k} = \tau_k\cdot \mu_{n,1},
\end{equation} precisely, if $\mu_{n,k+1}\leq g(a)$, then $$ \# I_{n,k} =
\mu_{n,1}\cdot (\tau_{k+1} -\tau_k).
$$

There is $\widetilde{a}<a$ such that for any $n\geq 1$ and $k$
such that for any $\mu_{n,k}\leq \widetilde{a}$ the sequence of
intervals $$ (0, \mu_{n,k})\stackrel{g_r}{\longleftarrow}
 (\mu_{n+1,2^n-k}, 1) \stackrel{g_l}{\longleftarrow} (\mu_{n+2,2^n-k},
\mu_{2,1})\stackrel{g_l}{\longleftarrow} \ldots
\stackrel{g_l}{\longleftarrow} (\mu_{n+2+m,2^n-k}, \mu_{2+m,1})
$$
does not contain any kink of $g$, where $m\geq 1$ is minimal
natural number such that $\mu_{2+m,1}\leq a$.

By Lemma~\ref{lema:2.18}, $$ \# (\mu_{n+1,2^n-k}, 1) =
\frac{1}{g'(1)}\cdot\# (0, \mu_{n,k}),
$$$$
\#(\mu_{n+2,2^n-k}, \mu_{2,1}) =\frac{1}{g'(v-)}\cdot
\#(\mu_{n+1,2^n-k}, 1)
$$ and there is $p\in \mathbb{R}_+$, independent on $\mu_{n,k}$,
such that \begin{equation}\label{eq:3.3}\# (0, \mu_{n,k}) =p\cdot
\# (\mu_{n+2+m,2^n-k}, \mu_{2+m,1}).\end{equation}

Denote by $n^*$ a minimal natural number such that
$\mu_{n^*,1}\leq \widetilde{a}$. Now for any $n\geq n^*$ and
$k\geq 1$ such that $\mu_{n,k}\leq \widetilde{a}$, we can
rewrite~\eqref{eq:3.3} as
$$\tau_2^{2+m}\cdot \# (0, \mu_{n+2+m,k}) =p\cdot \#
(\mu_{n+2+m,2^n-k}, \mu_{n+2+m,2^n}),$$ or, by~\eqref{eq:3.2},
$$\tau_2^{2+m}\cdot \tau_k =p\cdot (\tau_{2^n} -\tau_{2^n-k}).$$
Denote $q = \frac{p}{t_2^{2+m}}$. Then
\begin{equation}\label{eq:3.4}\tau_k =q\cdot (\tau_{2^n}
-\tau_{2^n-k}).\end{equation}

Suppose that $\mu_{n,k+1}\leq \widetilde{a}$. Then, plug $k+1$
instead of $k$ into~\eqref{eq:3.4} and obtain
$$
\tau_{k+1} =q\cdot (\tau_{2^n} -\tau_{2^n-k-1}).
$$
After the subtraction of the obtained equality and~\eqref{eq:3.4},
write
\begin{equation}\label{eq:3.5}
\tau_{k+1} -\tau_k = q\cdot(\tau_{2^n-k} -\tau_{2^n-k-1}).
\end{equation}

Suppose that $n\geq n^*$ and $s\geq 1$ is such that
$\mu_{s,2^{n}+1}\leq \widetilde{a}$. Then rewrite~\eqref{eq:3.5}
as
\begin{equation}\label{eq:3.6}
\# I_{s,k} = q\cdot \# I_{s,2^n-k-1} \end{equation} for all $k,\,
0\leq k\leq 2^{n}$. Since we can change $k$ to $2^n -k -1$
in~\eqref{eq:3.6}, then \begin{equation}\label{eq:3.7}\#
I_{s,2^n-k-1} = q\cdot \# I_{s,k}
\end{equation} %
for all $k,\, 0\leq k\leq 2^{n}$. Equalities~\eqref{eq:3.6}
and~\eqref{eq:3.7} imply that $q=1$ in~\eqref{eq:3.6}, i.e.
\begin{equation}\label{eq:3.8} \# I_{s,k} = \# I_{s,2^n-k-1}
\end{equation} for all $k,\, 0\leq k\leq 2^n-1$, all $s\geq 1$ and
$n\geq n^*$ such that $I_{s,2^n}\subseteq (0, \widetilde{a})$.

For any $i,\, 0\leq i<s$ write $$ \overline{I_{s-i,k}}
=\bigcup\limits_{j=0}^{2^i-1} \overline{I_{s,k\cdot 2^i+j}}
$$ and by~\eqref{eq:3.8} get $$\#
I_{s-i,k} = \# I_{s-i,2^{n-i}-k-1}
$$ for all $k,\, 0\leq k\leq 2^{n-i}-1$. Thus, the condition
$n\geq n^*$ and $I_{s,2^n}\subseteq (0, \widetilde{a})$ is not
important in~\eqref{eq:3.8}, i.e.
\begin{equation}\label{eq:3.9} \# I_{n,k} = \# I_{n,k+1}
\end{equation}
for all $n\geq 1$ and $k$ such that $I_{n,k+1}\subseteq (0,
\widetilde{a})$. Now lemma follows from~\eqref{eq:3.9} by
induction on $n$.
\end{proof}

We are now ready to prove Theorem~\ref{th:3}.

\begin{proof}[Proof of Theorem~\ref{th:3}]
It follows from Lemma~\ref{lema:3.2} that the conjugacy $h$ has
tangent $w$ on $\left( 0, \frac{w}{g(a)}\right)$, where $a$ is the
first kink of $g$.

Now Theorem~\ref{th:3} follows from Lemma~\ref{lema:2.17}.
\end{proof}

\subsection{Firm carcass maps}

Let a firm carcass map $g$ be fixed till the end of this section
and $n_0$ denote the same as in Section~\ref{sec:2.3}. We will
prove Theorem~\ref{th:2} in this section.

We will need the following technical fact.

\begin{remark}\label{rem:3.3}
For any $n\geq n_0$ and $i,\, j,\, k\in \{0,\ldots, 2^{n_0}-1\}$
we have that $$ \# I_{n+3n_0,2^{2n_0}\cdot i +2^{n_0}\cdot j +k} =
\frac{l_i\cdot l_j\cdot
l_k}{\left(\sum\limits_{p=1}^{2^{n_0}-1}l_p\right)^2} \cdot \#
I_{n,0}.
$$
\end{remark}

\begin{proof}
By Remark~\ref{rem:2.9} write $$ \# I_{n+3n_0,2^{2n_0}\cdot i
+2^{n_0}\cdot j +k} = l_k\cdot \# I_{n+3n_0,2^{2n_0}\cdot i
+2^{n_0}\cdot j }.
$$

By Remark~\ref{rem:2.12} write $$ \# I_{n+3n_0,2^{2n_0}\cdot i
+2^{n_0}\cdot j } = \frac{\# I_{n+2n_0,2^{n_0}\cdot i +j
}}{\sum\limits_{k=1}^{2^{n_0}-1}l_k}.
$$

Ones more by Remark~\ref{rem:2.9} simplify $$ \#
I_{n+2n_0,2^{n_0}\cdot i +j } = l_j\cdot \# I_{n+2n_0,2^{n_0}\cdot
i}
$$ and by Remark~\ref{rem:2.12} obtain $$
\# I_{n+2n_0,2^{n_0}\cdot i} = \frac{\#
I_{n+n_0,i}}{\sum\limits_{k=1}^{2^{n_0}-1}l_k}.
$$

Now our remark follows from Remark~\ref{rem:2.9}.
\end{proof}

\begin{remark}\label{rem:3.4}
Notice that for any $n\geq n_0$ and $k,\, 0\leq k<2^{n-1}-1$ we
have that $$ \frac{\# I_{n+1,2k+1}}{\# I_{n+1,2k}} =
\frac{1}{\delta_{k}} -1.
$$
\end{remark}

\begin{proof}
Indeed, $$ \frac{\# I_{n+1,2k+1}}{\# I_{n+1,2k}} = \frac{\#
I_{n,k} -\# I_{n+1,2k}}{\# I_{n+1,2k}} = \frac{1}{\delta_{n,k}} -1
= \frac{1}{\delta_{k}} -1.
$$
\end{proof}

\begin{lemma}\label{lema:3.5}
Assume that there is $a\in (0, 1)$ and a number $t\in \mathbb{N}$,
which is not a natural power of $2$, the equality
\begin{equation}\label{eq:3.10} \xi(\mu_{n,k}) = \mu_{n,tk},
\end{equation} whenever $\mu_{n,tk}\leq a$, holds. Then
$$\delta_0 = \delta_1 =\ldots =\delta_{2^{n_0}-1}.$$
\end{lemma}

\begin{proof} Suppose that $t = s\cdot 2^m$, where $s$ is odd. Thus,
by~\eqref{eq:3.10} and Remark~\ref{rem:2.3}, obtain that
\begin{equation}\label{eq:3.11} \xi(\mu_{n,k}) = \mu_{n-m,sk}.
\end{equation}

For any $k\geq 0$ such that $\mu_{n,tk}<a$ it follows
from~\eqref{eq:3.10} that
$$ \xi\circ [\mu_{n,k}, \mu_{n,k+1}] = [\mu_{n-m,sk},
\mu_{n-m,sk+s}],
$$ moreover, \begin{equation}\label{eq:3.12}
\xi(\mu_{n+1,2k+1}) = \mu_{n-m+1,s(2k+1)}. \end{equation} Thus, it
follows from~\eqref{eq:3.11} and~\eqref{eq:3.12} that
$$ \frac{\mu_{n,k+1}-\mu_{n+1,2k+1}}{\mu_{n+1,2k+1}-\mu_{n,k}} =
\frac{\mu_{n-m,sk+s}-\mu_{n-m+1,s(2k+1)}}{\mu_{n-m+1,s(2k+1)}
-\mu_{n-m,sk}}.
$$

By Notation~\ref{not:2.5}, we can rewrite the last equality as
\begin{equation}\label{eq:3.13}
\frac{\# I_{n+1,2k+1}}{\# I_{n+1,2k}} =
\frac{\sum\limits_{i=s}^{2s-1}\#
I_{n-m+1,2sk+i}}{\sum\limits_{i=0}^{s-1}\# I_{n-m+1,2sk+i}}.
\end{equation}

By Remark~\ref{rem:3.4} we can rewrite~\eqref{eq:3.13} as
\begin{equation}\label{eq:3.14}
\frac{1}{\delta_{k}} -1 = \frac{\sum\limits_{i=s}^{2s-1}\#
I_{n-m+1,2sk+i}}{\sum\limits_{i=0}^{s-1}\# I_{n-m+1,2sk+i}}.
\end{equation}

Notice that the right side of~\eqref{eq:3.14} contains $2s$
intervals, whose the second indices are the consequent numbers
\begin{equation}\label{eq:3.15} \{ 2sk, 2sk+1,\ldots,\, 2sk +2s
-1\}.
\end{equation}

For every $i\geq 0$ denote $k_i$ the number such that the
set~\eqref{eq:3.15} contains $i\cdot 2^{n_0}$. Next, denote $j_i,\
0\leq j_i\leq 2s-1$ such that $2sk_i +j_i = i\cdot 2^{n_0}$. Since
$s$ and $2^{\, n_0}$ do not have common divisors, then $$ \{ j_0,
j_1, \ldots, j_{s-1}\} = \{0, 2,\ldots, 2s-2\},
$$ moreover the set $j_i,\, i\geq 0$ is periodical with period
$s$.

Denote $k^*$ the minimal value $k = k_i\geq 0$ such that $j_i =
2s-2$. Then there is $w,\, 0< w\leq s$ such that plugging $k =
k^*$ into~\eqref{eq:3.14} transforms it to
\begin{equation}\label{eq:3.16}
\frac{1}{\delta_{k^*}} -1 = \frac{\sum\limits_{i=-1}^{s-2}\#
I_{n-m+1,2^{n_0}\cdot w -i}}{\sum\limits_{i=s-1}^{2s}\#
I_{n-m+1,2^{n_0}\cdot w -i}}
\end{equation}

Denote $k_i^{*} = k^* + i\cdot s\cdot 2^{\, n_0}$, where $i\geq
0$. Then the plug $k^*_i$ instead of $k^*$ into~\eqref{eq:3.16}
transforms it to \begin{equation}\label{eq:3.17}
\frac{1}{\delta_{k^* + i\cdot s\cdot 2^{\, n_0}}} -1 =
\frac{\sum\limits_{j=-1}^{s-2}\# I_{n-m+1,2^{n_0}\cdot (w+i\cdot
s\cdot 2^{\, n_0}) -j}}{\sum\limits_{j=s-1}^{2s}\#
I_{n-m+1,2^{n_0}\cdot (w+i\cdot s\cdot 2^{\, n_0}) -j}}.
\end{equation}

By Remarks~\ref{rem:2.11} and~\ref{rem:3.3} we can
rewrite~\eqref{eq:3.17} as
$$
\frac{1}{\delta_{k^*}} -1 = \frac{\frac{l_{i\cdot s}}{l_{i\cdot
s-1}}\cdot \sum\limits_{j=1}^{s-2} l_{i\cdot s}\cdot l_w\cdot
l_{2^{n_0-j}} +\sum\limits_{j=1}^{s-2} l_{i\cdot s}\cdot l_w\cdot
l_{2^{n_0-j}} }{\sum\limits_{j=s-1}^{2s}l_{i\cdot s}\cdot l_w\cdot
l_{2^{n_0-j}}},
$$ which can be cancelled to
$$
\frac{1}{\delta_{k^*}} -1 = \frac{\frac{l_{i\cdot s}}{l_{i\cdot
s-1}}\cdot \sum\limits_{j=1}^{s-2} l_{2^{n_0-j}}
+\sum\limits_{j=1}^{s-2} l_{2^{n_0-j}} }{\sum\limits_{j=s-1}^{2s}
l_{2^{n_0-j}}}
$$

Since the left hand side is independent on $i$, then so is  right
hand side. Thus, there is $q> 0$ such that $$\frac{l_{i\cdot
s}}{l_{i\cdot s-1}} =q$$ for all $i$. Since $s$ and $2^{n_0}$ are
pairwise prime, then for every $k,\, 0< k\leq 2^{n_{0}}$ we have
that $$ \frac{l_k}{l_{k-1}} =q,
$$ which means that $$l_0 =l_1 =\ldots = l_{2^{n_0}-1}.$$

Now Lemma follows from Remark~\ref{rem:3.4}.
\end{proof}

We are now ready to proof Theorem~\ref{th:2}.

\begin{proof}[Proof of Theorem~\ref{th:2}]
By Lemma~\ref{lema:3.5} obtain that $\delta_0 = \delta_1 =\ldots
=\delta_{2^{n_0}-1}$. Thus, by Lemma~\ref{lema:2.17}, the
conjugacy $h$, which satisfies~\eqref{eq:1.6}, is piecewise
linear.

Since any continuous solution $\psi$ of~\eqref{eq:1.4} can be
expressed by~\eqref{eq:2.1} from some continuous solution $\xi$
of~\eqref{eq:1.2}, then Theorem~\ref{th:2} follows from
Proposition~\ref{prop:1}, because $\xi$ is piecewise linear.
\end{proof}

\setlength{\unitlength}{1pt}

\pagestyle{empty}
\bibliography{Ds-Bib}{}

\begin{thebibliography}{10}

\bibitem{UMZh-2016}
V.~V. Kyrychenko and M.~V. Plakhotnyk,
``Topological Conjugate Piecewise Linear Unimodal Mappings of an Interval Into
  Itself'',
{\it Ukrainian Mathematical Journal},
Vol. {\bf 68}, No~2, pp.~242--252, 2016 (in Ukrainian. See English translation
  at https://doi.org/10.1007/s11253-016-1221-6).

\bibitem{Plakh-Arx-Odesa}
M. Plakhotnyk,
``{Self semi conjugations of Ulam's Tent-map}'',
{\it ArXiv}, eprint: 1703.09753, math.DS, mar. 2017.

\bibitem{Plakh-Arx-Derivat-6}
M. {Plakhotnyk},
``{The derivative of the conjugacy for the pair oftent-like maps from an
  interval into itself}'' (Version 6),
{\it ArXiv}, eprint: 1707.09874, math.DS, Dec. 2017.

\bibitem{Ulam-1964-a}
S. Ulam,
{\it Sets, Numbers, and Universes}, edited by W. A. Beyer, J, Mycielski, and
  G.-C. Rota., Cambridge, Massachusetts: The MIT Press, 1974.

\end{thebibliography}
\bibliographystyle{makar}

\newpage
\tableofcontents

\end{document}